\documentclass[12pt]{amsart}
\usepackage[utf8]{inputenc}

\usepackage{latexsym}
\usepackage{amsmath}
\usepackage{amssymb, amscd, amsthm, mathtools}
\usepackage[all]{xy}
\usepackage{multirow}
\usepackage{xcolor} 
\usepackage{tikz}
\usepackage{tikz-cd}
\usepackage[shortlabels, inline]{enumitem}
\SetEnumerateShortLabel{a}{\textup{(\alph*)}}
\SetEnumerateShortLabel{A}{\textup{(\Alph*)}}
\SetEnumerateShortLabel{1}{\textup{(\arabic*)}}
\SetEnumerateShortLabel{i}{\textup{(\roman*)}}
\SetEnumerateShortLabel{I}{\textup{(\Roman*)}}

\newcommand{\Z}{\ensuremath{\mathbb Z}}

\newcommand{\Q}{\ensuremath{\mathbb Q}}
\newcommand{\cO}{{\mathcal O}}

\newtheorem{theorem}{Theorem} 

\newtheorem{proposition}[theorem]{Proposition}
\newtheorem{Conjecture}[theorem]{Conjecture} 
\newtheorem*{remark}{Remark} 
\newtheorem{lemma}[theorem]{Lemma}
\newtheorem{corollary}[theorem]{Corollary}

\begin{document}
\title[Divisibility of class numbers]{{Divisibility of class numbers of quadratic fields and a conjecture of Iizuka}}
\author{Yi Ouyang$^{1,2}$ and Qimin Song$^1$}

\address{$^1$School of Mathematical Sciences, Wu Wen-Tsun Key Laboratory of Mathematics,   University of Science and Technology of China, Hefei 230026, Anhui, China}

\address{$^2$Hefei National Laboratory, University of Science and Technology of China, Hefei 230088, China}
	
\email{yiouyang@ustc.edu.cn}	
\email{sqm2020@mail.ustc.edu.cn}
	
\thanks{Partially supported by NSFC (Grant No. 12371013) and  Innovation Program for Quantum Science and Technology (Grant No. 2021ZD0302902).}
	
	
\date{}
\maketitle

\begin{abstract} 
	
Assume $x,\ y,\ n$ are positive integers and $n$ is odd.	
In this note, we  show that  the class number of the imaginary quadratic field $\Q(\sqrt{x^{2}-y^{n}})$ is divisible by $n$ for fixed $x, n$ if $\gcd(2x,y)=1$ and  $y>C$ where $C$ is a constant depending only on $x$ and $n$. Based on this result, for any odd integer $n$ and any positive integer $m$,  we construct an infinite family of $m+1$ successive imaginary quadratic fields  $\Q(\sqrt{d})$, $\Q(\sqrt{d+1^{2}})$,  $\cdots$, $\Q(\sqrt{d+m^{2}})\ (d\in \Z)$ whose class numbers are all divisible by $n$.
\end{abstract}

\bigskip	
The $n$-divisibility of class numbers of quadratic fields for a given positive integer $n$ is well studied in the literature (see for example ~\cite{ref1, ref5, ref8}). 
Further more, Iizuka~\cite{ref4} studied the $p$-divisibility of class number of pairs of quadratic fields.  He constructed an infinite family of pairs of imaginary quadratic fields $\Q(\sqrt{d})$ and $\Q(\sqrt{d+1})$ with $d\in \Z$ whose class numbers are divisible by $3$. Based on  this result, he  posed the following conjecture:
	
\begin{Conjecture}
For any prime number $p$ and any positive integer $m$, there exists an infinite family of $m+1$ successive imaginary quadratic fields $\Q(\sqrt{d})$, $\Q(\sqrt{d+1})$, $\cdots$, $\Q(\sqrt{d+m})$ with $d\in \Z$ whose class numbers are divisible by $p$.
\end{Conjecture}
	
In~\cite{ref2}, Hoque generalized Iizuka's Conjecture by replacing odd prime $p$ by odd integer $n$. For an odd prime $p$, he obtained results about the $n$-divisibility of the class number of $\Q(\sqrt{p^{2}-l^{n}})$ under some module conditions. Based on this result, he constructed an infinite family of quadruples of imaginary quadratic fields
	\begin{equation*}
		\Q(\sqrt{d}),\ \Q(\sqrt{d+1}),\ \Q(\sqrt{d+4}),\ \Q(\sqrt{d+4p^{2}})
	\end{equation*}
with $d\in \Z$ whose class numbers are divisible by $n$.
	
In~\cite{ref9}, for $k$ a cubic-free positive integer such that $k\equiv 1 \pmod{9}$ and $\gcd(k,7\times 571)=1$, Chattopadhyay and  Muthukrishnan constructed an infinite family of  triples of imaginary quadratic fields $\Q(\sqrt{d})$, $\Q(\sqrt{d+1})$, $\Q(\sqrt{d+k^{2}})$ with $d\in \Z$ whose class number are divisible by $3$.
	
In this note, we study the class number of the imaginary quadratic field
  \[ L=L_{x,y,n}:=\Q(\sqrt{x^2-y^n}),\ x,y,z\in \Z_+,\ 2\nmid n,\ x^2<y^n. \] 
Our main result is 
\begin{theorem}  \label{thm:22} Fix $x$ and $n$. Suppose    $\gcd(2x,y)=1$ and $y>C$ where $C$ is a constant depending only on $x$ and $n$. Then  the class number of $L_{x,y,n}= \Q(\sqrt{x^{2}-y^{n}})$ is divisible by $n$. 
 \end{theorem}

Based on this result, we have a weaker version of Iizuka's Conjecture: 
\begin{theorem}  \label{thm:24} 
For any odd integer $n\geq3$ and any positive integer $m$, there is an infinite family of $m+1$ successive imaginary quadratic fields $\Q(\sqrt{d})$, $\Q(\sqrt{d+1^{2}})$, $\cdots$, $\Q(\sqrt{d+m^{2}})$ with $d\in \Z$ whose class numbers are all divisible by $n$.
\end{theorem}

For an integer $d$ which is not a square, we denote by $h(d)$  the class number of the quadratic field $\Q(\sqrt{d})$.
We need the following results:
\begin{proposition}[Cohn~\cite{ref6}] \label{thm:12}
Assume $n\geqslant 3$ and $V\geqslant 3$ are odd integers. Then $n\mid h(1-V^{n})$ except for the case $(V,n)=(3,5)$.
\end{proposition}
	\begin{proposition}[Hoque-Chakraborty~\cite{ref10}] \label{thm:13} Let $m>1$ and $p$ be odd integers, and $n$ be any positive integer. Let $d$ be the square-free part of $-(3^{m}p^{2n}+r)$ with $r\in \{-2, 4\}$. Then $3\mid h(d)$.
	\end{proposition}
	

\begin{lemma} \label{lemma:21}
For $n>1$ an odd integer, $x$ and $y$ positive integers such that $x^{2}<y^{n}$, write $y^{n}-x^{2}=dt^{2}$ with $d$ square-free. Then for any $k>0$, there exists a constant $K(x,n,k)>0$ such that $d>k$ if $y>K(x,n,k)$.
\end{lemma}
\begin{proof}
For any positive integer $D$, the curve $Y_{D}: DT^{2}+x^{2}=Y^{n}$ is an algebraic curve of genus $> 0$. By Siegel's Theorem, $Y_{D}$ has only finite many integral points. Consider the set 
\begin{equation*}
A_{k}=\{y_{P}\mid P:(t_{P},y_{P}) \in Y_{D}(\Z)\quad \text{for\ some}\ 1\leqslant D \leqslant k\}.
\end{equation*}
Then $A_{k}$ is a finite subset of $\Z$. Take $K(x,n,k)$ to be the largest integer in $A_{k}$. We have $d>k$ if $y>K(x,n,k)$ by construction.
\end{proof}

\begin{proof}[Proof of Theorem~\ref{thm:22}] Let $d$ be the square-free part of $y^{n}-x^{2}$ and write $x^{2}-y^{n}=-dt^{2}$. We have $(x-t\sqrt{-d})(x+t\sqrt{-d})=y^{n}$. Since $\gcd(2x,y)=1$, we have $(x-t\sqrt{-d})$ is coprime with $(x+t\sqrt{-d})$. Thus $(x-t\sqrt{-d})$ and $(x+t\sqrt{-d})$ are both  $n$-th power for some integral ideals. Write $(x-t\sqrt{-d})=\alpha^{n}$ where $\alpha$ is an integral ideal.  Suppose $m$ is the order of $[\alpha]$  in the class group of $L=L_{x,y,n}$. Then $n=ms$.  To show $n\mid h(x^2-y^n)$, it suffice to show that $[\alpha]$ is an element of order $n$, equivalently $s=1$. We show this is the case when $y$ is  large enough.
		
		\bigskip
		
\noindent (I) If $x^{2}-y^{n}\equiv 2,3 \pmod{4}$, then $\cO_{L}=\Z[\sqrt{-d}]$. Denote $\alpha^{m}=(a+b\sqrt{-d})$, thus we have
		\begin{equation} \label{eq:1}
			(a+b\sqrt{-d})^{s}=(x-t\sqrt{-d}).
		\end{equation}
We shall use Lemma~\ref{lemma:21} often. First assume
 \[ y>K_1= K(x,n,3)\ \Longrightarrow\ d>3. \]  
Then the units of $\cO_{L}$ is $\{\pm 1\}$. Assume $a>0$ and compare the real parts of both sides of \eqref{eq:1}, we get
		\begin{equation}  \label{eq:2}
			\pm x=a\sum_{j=0}^{\frac{s-1}{2}}\binom{s}{2j}a^{s-2j-1}b^{2j}(-d)^{j}.
		\end{equation}	  
Thus $\pm x \equiv a^{s} \mod d$ and $a\mid x$. Assume 
 \[ y>K_2= K(x,n,x+x^n)\ \Longrightarrow\ d>x+x^n> x+a^s. \]   
Then \eqref{eq:2} implies that  $x=a^{s}$ and that
		\begin{equation}  \label{eq:3}
			\sum_{j=1}^{\frac{s-1}{2}}\binom{s}{2j}a^{s-2j}b^{2j}(-d)^{j}=0
		\end{equation}
if $s>1$.  Note that $\gcd(d,x)=\gcd(d,a)=1$, hence  $d\mid \binom{s}{2}b^{2}$. Assume 
\[ y>K_3= K(x,n,n!)\ \Longrightarrow\ d>n!. \]
Then there exists a prime factor $q$ of $d$ prime to $\binom{s}{2}$. Thus we have $q\mid b$. Compute the $q$-adic values of terms in RHS of \eqref{eq:3}, the $j=1$ term is the smallest, thus \eqref{eq:3} can not hold.   Hence we have $s=1$ if
 \[ y>C=\max\{K(x,n,3), K(x,n,x+x^n), K(x,n, n!)\}.\]

		\bigskip
\noindent (II) If $x^{2}-y^{n}\equiv 1 \pmod{4}$, then $\cO_{L}=\Z[\frac{1+\sqrt{-d}}{2}]$. Denote $\alpha^{m}=(a+b\frac{1+\sqrt{-d}}{2})$, we have
		\begin{equation}
			\left(a+b\frac{1+\sqrt{-d}}{2}\right)^s=(x-t\sqrt{-d}).
		\end{equation}
		Assume that $2a+b\geq 0$, take real part of $(4)$ we have 
		\begin{equation}
			\pm 2^{s}x=(2a+b)\sum_{j=0}^{\frac{s-1}{2}}\binom{s}{2j}(2a+b)^{s-2j-1}b^{2j}(-d)^{j}.
		\end{equation}
By the same argument as in (I), if 
 \[ y>C=\max\{K(x,n, 2^n x+(2^n x)^n), K(x,n.n!)\},\]		
then $s=1$.		
\end{proof}
\begin{remark}
Apply the same strategy, supposed that $\gcd(x,y)=1$ and $x^{2}<4y^{n}$, one can show that $n\mid h(x^2-4y^n)$  for fixed $x$ and $n$ if $y>C$ for some constant $C$ depending only on $x$ and $n$.
\end{remark}

\begin{proof}[Proof of Theorem~\ref{thm:24}]
By Theorem~\ref{thm:22}, for $1\leqslant x \leqslant m$, there exist $C(x,n)>0$ such that $n\mid h(x^2-y^n)$ if $y>C(x,n)$ and $\gcd(2x,y)=1$. Take $N=\max\{C(x,n)\mid 1\leqslant x\leqslant m\}$. Then  $n\mid h(x^{2}-y^{n})$  if $y>N$ and $\gcd(2x,y)=1$ for $1\leqslant x\leqslant m$. Take 
  \[ y=((m+1)!)^{nl}-1,\quad d=\Bigl(1-\bigl((m+1)!\bigr)^{nl}\Bigr)^{n}=(-y)^n \] 
with $l$ large enough such that  $y>N$. Then $\gcd(2x,y)=1$ for $1\leq x\leq m$. By Proposition~\ref{thm:12}, then
 \[ n\mid h(-d);\]
by Theorem~\ref{thm:22}, then 
 \[ n\mid h(d+x^2)=h(x^2-y^n)\ \text{for}\ 1\leq x\leq m. \]
Thus we get an infinite family of $m+1$ successive imaginary quadratic fields $\Q(\sqrt{d})$, $\Q(\sqrt{d+1^{2}})$, $\cdots$, $\Q(\sqrt{d+m^{2}})$ with $d\in \Z$ whose class numbers are all divisible by $n$.
\end{proof}

Theorem~\ref{thm:22} also has the following corollaries:
\begin{corollary}  \label{thm:25} 
Let $n>1$ and $m$ be odd integers. Then there exist infinite pairs of quadratic fields $\Q(\sqrt{d})$ and $\Q(\sqrt{d+m})$ whose class numbers are both divisible by $n$.
\end{corollary}
\begin{proof}
Denote $m=2k-1=k^2-(k-1)^{2}$, and take
 \[ d=(k-1)^{2}+\bigl(1-(k!)^{l}\bigr)^{n}.\]
By Theorem~\ref{thm:22}, we have $n\mid h(d)$ and $n\mid h(d+m)$ if  $l$ is large enough.
\end{proof}
\begin{remark} 
Note that one can replace $m$ by any integer of form $2^{2t}k$ with $k$ odd. This partly recovers a result of Xie-Chao in  \cite{ref7}, where $m$ can be taken as any integer using Yamamoto's construction~\cite{ref8}.
\end{remark}

\begin{corollary}
There exist infinite triples of quadratic fields $\Q(\sqrt{d})$, $\Q(\sqrt{d+1})$, $\Q(\sqrt{d+3})$ with $d\in \Z$ whose class number are all divisible by $3$.
\end{corollary}
\begin{proof}
Let $p>3$  be an odd prime and take $d=1-3^{3k}p^{6t}=1-(3^{k}p^{2t})^{3}$. According to Proposition~\ref{thm:13}, the class number of $\Q(\sqrt{d+1})$ is divisible by $3$. Since $\gcd(2,3p)=1$, we have $3\mid h(d)$ and $3\mid h(d+3)$ if  $k$ and $t$ are large enough by Theorem~\ref{thm:22}. This completes the proof.  
\end{proof}

\end{document}